\documentclass[a4paper,11pt]{amsart}
\usepackage{amsmath,amsthm,amssymb}
\usepackage{tikz}
\usepackage{mathrsfs}
\usepackage{enumerate}
\usepackage{algorithm,algorithmic}

\DeclareMathOperator{\Mlt}{Mlt}
\DeclareMathOperator{\Inn}{Inn}

\DeclareMathOperator{\Aut}{Aut}
\DeclareMathOperator{\id}{id}
\DeclareMathOperator{\ad}{ad}
\DeclareMathOperator{\End}{End}

\newtheorem{theorem}{Theorem}[section]
\newtheorem{proposition}[theorem]{Proposition}
\newtheorem{lemma}[theorem]{Lemma}

\newtheorem{problem}[theorem]{Problem}

\begin{document}

\title{On centerless commutative automorphic loops}
\author{G\'abor P. Nagy} 
\email{nagyg@math.u-szeged.hu}
\address{Bolyai Institute, University of Szeged, Aradi v\'ertan\'uk tere 1, H-6720 Szeged (Hungary)}
\address{MTA-ELTE Geometric and Algebraic Combinatorics Research Group, P\'azm\'any P. s\'et\'any 1/c, H-1117 Budapest (Hungary)}


\begin{abstract}
In this short paper, we survey the results on commutative automorphic loops and give a new construction method. Using this method, we present new classes of commutative automorphic loops of exponent $2$ with trivial center.
\end{abstract}

\keywords{Commutative automorphic loop, Lie algebra of characteristic $2$, nuclear extension}

\subjclass[2010]{20N05}

\maketitle

\section{Introduction}

The set $Q$ endowed with the binary operation $x\cdot y$ is a \textit{quasigroup} if the equation $x\cdot y=z$ has a unique solution provided two of the variables $x,y,z$ are known. The solution is denoted by $x=z/y$ and $y=x\backslash z$. The maps $R_y:x\mapsto x\cdot y$ and $L_x:y\mapsto x\cdot y$ are the \textit{multiplication maps} of $(Q,\cdot)$; these are permutations of $Q$. The \textit{multiplication group} $\Mlt(Q)$ is generated by the multiplication maps. Quasigroups with a unit element $1$ are called \textit{loops;} $L_1=R_1=\id$ holds. Let $Q$ be a loop. The stabilizer subgroup of $1$ in $\Mlt(Q)$ is the \textit{group of inner mappings} $\Inn(Q)$. The loop $Q$ is an \textit{automorphic loop} if $\Inn(Q)\leq \Aut(Q)$. In particular, every group is an automorphic loop. The study of automorphic loops was initiated by Bruck and Paige \cite{BP} and it is in the very focus of recent research of the theory of loops and quasigroups. 

Automorphic loops are power associative (\cite{BP}) and have the inverse property (\cite{JKNV}). Let $Q$ be a finite automorphic loop. For any prime $p$, all elements of $Q$ have $p$-power order iff $|Q|$ is a power of $p$ (\cite{KGN,KKPV}). We call such loops automorphic $p$-loops. Automorphic $p$-loops are solvable (\cite{KGN,KKPV}). Automorphic loops of finite odd order are solvable (\cite{KKPV}). Automorphic loops of order $p^2$ are associative (\cite{Csorgo2,KKPV}). The class of finite commutative automorphic loops is interesting on its own. Most importantly, finite commutative automorphic loops are solvable (\cite{KGN}). For an odd prime $p$, commutative automorphic $p$-loops are nilpotent (\cite{Csorgo1,JKV3}). 

By now, there are many constructions showing that the above results are sharp in some sense. For any odd prime $p$, we have examples of nonnilpotent automorphic $p$-loops (\cite{JKV2,JKV3}). For any $n\geq 3$, there are examples of nonnilpotent commutative automorphic loops of order $2^n$ (\cite{JKV1}). If $p<q$ are primes and $p$ divides $q-1$ or $q+1$, then we have examples of nonassociative commutative automorphic loops of order $pq$ (\cite{Drapal}). 

Finally, we mention that Lagrange and Cauchy Theorems are known to hold for commutative automorphic loops and automorphic loops of odd order. The main open problems in the theory of finite automorphic loops are the Sylow Theorems and the existence of nonassociative finite simple automorphic loops. 

In this paper, we investigate finite commutative automorphic loops of exponent $2$. Our main tool is the Lie ring method of \cite{KGN}, which was developed further in \cite{KKPV}. Our main result is to give a wide class of finite commutative automorphic loops of exponent $2$ with trivial center.

\section{Two constructions}

In this section, we survey the properties of the two main construction methods of automorphic loops, namely the Lie ring method and the nuclear semidirect product method. 

Let $(Q,+,[.,.])$ be a Lie ring and assume that 
\begin{align} 
&\mbox{for any $x\in Q$, the maps $y\mapsto y\pm [x,y]$ are invertible.} \label{eq:W1}\tag{$W_1$} 
\end{align}
We define the operation 
\begin{equation} \label{eq:Lieringloop}
x\circ y= x+y-[x,y]
\end{equation}
on $Q$. By \cite[Lemma 5.1]{KKPV}, $(Q,\circ)$ is a loop. The Lie ring ideals correspond to normal subloops, in particular, $[Q,Q]$ corresponds to the commutator-associator subloop of $(Q,\circ)$. Define the further properties of the Lie ring $(Q,+,[.,.])$:
\begin{align} 
& [[x,y],[z,y]]=0 \mbox{ for all $x,y,z\in Q$.} \label{eq:W2}\tag{$W_2$} \\
& \mbox{$(Q,\circ)$ is an automorphic loop.} \label{eq:W2-}\tag{$W_2^-$} \\
& [[x,y],[z,w]]=0 \mbox{ for all $x,y,z,w\in Q$.} \label{eq:W2+}\tag{$W_2^+$} 
\end{align}

Clearly, \eqref{eq:W2+} implies \eqref{eq:W2}. By \cite[Proposition 5.2]{KKPV}, \eqref{eq:W2} implies \eqref{eq:W2-}. The Lie ring $(Q,+,[.,.])$ is \textit{uniquely 2-divisible} if the map $x\mapsto x+x$ is invertible. It was shown in \cite[Lemma 5.8]{KKPV}, that for uniquely 2-divisible Lie rings, \eqref{eq:W2-} implies \eqref{eq:W2+}, that is, \eqref{eq:W2}, \eqref{eq:W2+} and \eqref{eq:W2-} are equivalent. However, up to the author's knowledge, the situation is unknown in the general case. We formulate the problem for Lie rings of characteristic $2$.

\begin{problem} \label{prob:lie}
Let $(Q,+,[.,.])$ be a Lie ring of characteristic $2$. 
\begin{enumerate}[(a)]
\item Does \eqref{eq:W2-} imply \eqref{eq:W2}? 
\item Does \eqref{eq:W2} imply \eqref{eq:W2+}?
\end{enumerate}
\end{problem}

Notice that in this formulation of (b), no reference is made to property \eqref{eq:W1}. In fact, we think that this question is interesting for general Lie rings of characteristic $2$. Moreover, it is rather natural to restrict (b) to \textit{nilpotent} Lie rings of characteristic $2$; then \eqref{eq:W1} is automatically fulfilled. Using the Lie algebra database of the GAP package \textsf{LieAlgDB} \cite{LieAlgDB}, one can show that no nilpotent counterexample of dimension at most $9$ exists, cf. Proposition \ref{prop:dim9} in the appendix.


\medskip

Another important observation is the following. If the answers to both questions (a) and (b) are affirmative, then any commutative automorphic loop of exponent $2$, which is constructed from a Lie ring of characteristic $2$ turns out to be associative modulo its middle nucleus. Indeed, by \cite[Proposition 5.2]{KKPV}, $a\in N_\mu(Q)$ iff $[[x,y],a]=0$ for all $x,y\in Q$. That is, \eqref{eq:W2+} implies $[z,w] \in N_\mu(Q)$ for all $z,w \in Q$, and, the factor $Q/N_\mu(Q)$ is an elementary Abelian $2$-group. This fact justifies the importance of the class of commutative automorphic loops which are nuclear semidirect product of elementary abelian $2$-groups. For the proofs of the following facts the reader is referred to \cite{HoraJed}.

Let $H$, $K$ be abelian groups and $\Phi=\{ \varphi_{i,j} \in \Aut(K) \mid i,j \in H\}$ a family of automorphisms such that 
\begin{align}
\varphi_{i,j} &= \varphi_{j,i} \\
\varphi_{0,i} &= \id_K \\
\varphi_{i,j} \circ \varphi_{k,n} &= \varphi_{k,n} \circ \varphi_{i,j} \\
\varphi_{i,j+k}\circ \varphi_{j,k} &= \varphi_{k,i+j}\circ \varphi_{i,j} \\
\varphi_{i,j+k} + \varphi_{j,i+k} + \varphi_{k,i+j} &= \id_K + 2 \cdot \varphi_{i,j,k}
\end{align}
hold for all $i,j,k\in H$. Then, the underlying set $Q=K\times H$ endowed with the operation 
\begin{equation} \label{eq:nucsemi}
(a,i) \star (b,j)=(\varphi_{i,j}(a+b),i+j)
\end{equation}
is a commutative automorphic loop. Moreover, $K$ is a normal subloop contained in the middle nucleus, $H\leq Q$, and $Q/K\cong H$. Conversely, if $Q$ is an arbitrary commutative automorphic loop, $K$ an abelian normal subloop contained in $N_\mu(Q)$, $H$ an abelian subloop isomorphic to $Q/K$, then $Q$ can be constructed by \eqref{eq:nucsemi}. It is worth mentioning that nonsplitting examples of order $32$ can be constructed (see Proposition \ref{prop:nonsplit} in the appendix). We remark further that if $Q$ is a commutative automorphic loop of exponent $2$ and $N_\mu(Q)$ has index $2$ then $Q$ splits. 

Let $H,K$ be elementary abelian $2$-groups. \cite[Proposition 1.4]{HoraJed} implies that the center of $Q$ is
\[\{(a,i) \mid \varphi_{j,k}(a)=a \mbox{ and } \varphi_{i,j}=\id_K \mbox{ for all } j,k \in H\}.\]
In \cite{JKV1}, the authors describe completely the case $|H|=2$, that is, when the middle nucleus has index $2$. In this case, $\Phi$ consists of $\id_K$ and a single nontrivial automorphism $\varphi=\varphi_{1,1}$. (In order to see this, use \cite[Proposition 2.7]{JKV1} with the extra assumption that $Q$ has exponent $2$.) The construction of \cite[Section 3.2]{JKV1} gives a commutative automorphic loop of exponent $2$ with trivial center and order $2^n$ for any $n\geq 3$. 

At the end of this section we make a closer look at the construction of \cite[Theorem 4.3]{HoraJed} and show that it has a nontrivial center. Let $X$ be a subset of $\End(K)$ satisfying $X^2=0$. The group $G=\langle \id_K+x \mid x\in X \rangle_{\Aut(K)}$ is an elementary abelian $2$-group and $\varphi:H^2\to G$, $\varphi:(i,j)\mapsto \varphi_{i,j}$ a symmetric bilinear mapping. Since $G$ has a common fixed point $a\in K$, the center of the resulting commutative automorphic loop is nontrivial.

\section{Centerless commutative automorphic loops}

For vector spaces $K,H$, we identify the subspaces $K\oplus 0$ and $0\oplus H$ of $K\oplus H$ with $K$ and $H$, respectively.

\begin{lemma} \label{lm:lie}
Let $H,K$ be vector spaces over $\mathbb{F}_2$, and $\beta:H\to \End(K)$ a linear map. Define the bracket
\begin{equation}
[a\oplus i,b\oplus j] = (\beta(j)a+\beta(i)b)\oplus 0
\end{equation} 
on $Q=K\oplus H$. Then $(Q,+,[.,.])$ is a Lie algebra over $\mathbb{F}_2$ if and only if $\beta(i)$, $\beta(j)$ commute for all $i,j\in H$. Moreover, $[Q,Q]\subseteq K$ and $[[Q,Q],[Q,Q]]=0$ hold. 
\end{lemma}
\begin{proof}
The nontrivial part of the proof is the Jacobi identity which is equivalent to $\beta(i)\beta(j)+\beta(j)\beta(i)=0$ by a straightforward calculation. 
\end{proof}

\begin{proposition} \label{prop:aloop}
Let $H,K$ be vector spaces over $\mathbb{F}_2$, and $\beta:H\to \End(K)$ a linear map. Define the operation
\begin{equation} \label{eq:aloop}
(a\oplus i) * (b\oplus j) = (a+b+\beta(j)a+\beta(i)b)\oplus (i+j)
\end{equation} 
on $Q=K\oplus H$. Assume that for all $i,j\in H$
\begin{enumerate}[(i)]
\item the endomorphisms $\beta(i)$, $\beta(j)$ commute, 
\item and $\id +\beta(i)$ is invertible.
\end{enumerate}
Then $(Q,*,0)$ is a commutative automorphic loop of exponent $2$ with the following properties:
\begin{enumerate}[(a)]
\item $K\cap Z(Q) = \{a\in K \mid \beta(j)a=0 \mbox{ for all } j\in H\}$.
\item $H\cap Z(Q) = \{i \in H \mid \beta(i)\beta(j)=0 \mbox{ for all } j\in H \}$.
\item $Z(Q)= (K\cap Z(Q)) \oplus (H\cap Z(Q))$. 
\end{enumerate}
In particular, if $\beta$ is injective and at least one $\beta(j)$ is invertible then $Q$ has trivial center.
\end{proposition}
\begin{proof}
Since the $\beta(i)$'s commute, $(Q,+,[.,.])$ is a Lie algebra with $[[Q,Q],[Q,Q]]=0$ by Lemma \ref{lm:lie}. For fixed $a,b,b',i,j,j'$, 
\[\id+\ad_{a\oplus i}:b\oplus j\mapsto b'\oplus j' \]
holds iff $j=j'$ and $(\id +\beta(i))b=b'+\beta(j')a$. That is, $\id+\ad_{a\oplus i}$ is invertible iff $\id+\beta(i)$ is invertible. This means that $(Q,+,[.,.])$ satisfies \eqref{eq:W1} and $(Q,*,0)$ is a commutative automorphic loop of exponent $2$ by \cite[Proposition 5.2]{KKPV}. We intend to show the properties of $Z(Q)$ by referring to \cite[Proposition 1.4]{HoraJed}. In order to do that, we must construct $(Q,*)$ as nuclear semidirect product. Put 
\[\varphi_{i,j}=(\id +\beta(i+j))^{-1}(\id+\beta(i))(\id+\beta(j)) \]
and
\[u(a\oplus i) = (\id +\beta(i))a \oplus i.\]
Then, using \eqref{eq:aloop} and $\beta(i)\beta(j)=\beta(j)\beta(i)$
\begin{align*}
u(a\oplus i) * u(b\oplus j) &= ((\id +\beta(i))a \oplus i) * ((\id +\beta(j))b \oplus j) \\
&= [(\id+\beta(i))(\id+\beta(j))(a+b)]\oplus(i+j)\\
&= u(\varphi_{i,j}(a+b) \oplus (i+j)).
\end{align*}
Comparing this with \eqref{eq:nucsemi}, we see that $u$ is an isomorphism between $(Q,*)$ and the nuclear semidirect product $(Q,\star)$, associated to $\Phi=\{\varphi_{i,j} \mid i,j\in H\}$. Now, 
\cite[Proposition 1.4]{HoraJed} implies (a), (b) and (c), since $\varphi_{i,j}=\id$ iff $\beta(i)\beta(j)=0$. The last claim follows easily. 
\end{proof}

We finish the paper with two explicit examples of maps $\beta:H\to \End(K)$ which give rise to commutative automorphic loops of exponent $2$ with trivial nucleus. 

\medskip

\textit{Example 1.} Let $K$ be a field of characteristic $2$ and $\delta:H \hookrightarrow (K,+)$ an injective homomorphism such that $1\not \in \mathrm{Im}(\delta)$. Define $\beta:H\to \End_{\mathbb{F}_2}(K)$ by
\[\beta(i)(a)=\delta(i)a.\]
Then, $\beta$ is injective and $\id+\beta(i)$ is invertible for all $i\in H$. 

\medskip

The next example is a special case of Example 1.

\textit{Example 2.} Let $K$ be a field of characteristic $2$, $H$ a proper subfield and $\sigma \in K\setminus H$. Define $\beta:H\to \End_{\mathbb{F}_2}(K)$ by
\[\beta(i)(a)=\sigma i a.\]
Then, $\beta$ is injective and $\id+\beta(i)$ is invertible for all $i\in H$. 

\appendix

\section{Low dimensional nilpotent Lie algebras}
\label{app:low}

The Lie algebra database of the GAP package \textsf{LieAlgDB} \cite{LieAlgDB} contains the complete list of nilpotent Lie algebras with dimension at most $9$ over $\mathbb{F}_2$. We used this package for an exhaustive search for low dimensional counterexamples of Problem \ref{prob:lie}.

\begin{proposition}[Computational result] \label{prop:dim9}
Let $(Q,[.,.])$ be a nilpotent Lie algebra over $\mathbb{F}_2$ with $\dim Q\leq 9$. Then $Q$ satisfies either all or none of \eqref{eq:W2}, \eqref{eq:W2+} and \eqref{eq:W2-}. 
\end{proposition}
\begin{proof}
The basic idea is to test $Q$ first for \eqref{eq:W2}. If \eqref{eq:W2} holds, then we test for \eqref{eq:W2+}, in the negative case $Q$ would be a counterexample to (b). If \eqref{eq:W2} does not hold, then we test for \eqref{eq:W2-}, which would yield counterexample to (a) in the positive case. One can accelerate the test with the following tricks.
\begin{enumerate}[(i)]
\item Let $B$ be a basis of $Q$ over $\mathbb{F}_2$. \eqref{eq:W2} holds iff 
\[[[x,y],[z,y]]=[[x,y],[z,w]]=0\]
for all $x,y,z,w\in B$. 
\item \eqref{eq:W2+} holds iff the length of the Lie derived series of $Q$ is at most $3$.
\item Construct the loop $(Q,\circ)$ by \eqref{eq:Lieringloop}. Let $R$ be the right section of $(Q,\circ)$ and let $S$ be a generating set of $\Inn(Q,\circ)$. Then, \eqref{eq:W2-} holds iff $R=s^{-1}Rs$ for all $s\in S$. 
\end{enumerate}
\end{proof}

Since the number of nilpotent Lie algebras of dimension $5$ over $\mathbb{F}_2$ is $9$, the following example can be found. 

\begin{proposition}[Computational result] \label{prop:nonsplit}
There exists a nilpotent Lie algebra $(Q,[.,.])$ of dimension $5$ over $\mathbb{F}_2$ such that the corresponding loop $(Q,\circ)$ satisfies the following:
\begin{enumerate}[(i)]
\item $(Q,\circ)$ is a commutative automorphic loop of exponent $2$.
\item The index of the middle nucleus is $4$.
\item $Q$ is does not split nuclearly, that is, there are no subgroups $H, K$ such that $K$ is normal in $Q$, $K\leq N_\mu(Q)$, $Q=HK$ and $H\cap K=\{1\}$. \qed
\end{enumerate}
\end{proposition}

%
%
%

\subsection*{Acknowledgement} The author would like to thank the referee for many valuable remarks, and Michael Kinyon (Denver) for a motivating discussion on low dimensional Lie algebras. The publication is supported by the European Union and co-funded by the European Social Fund. Project title: \textit{Telemedicine-focused research activities on the field of Matematics, Informatics and Medical sciences.} Project number: TAMOP-4.2.2.A-11/1/KONV-2012-0073.

\end{document}